\documentclass[12pt]{amsproc}

\usepackage{amsmath, amsthm, amssymb, amsfonts, amscd, graphicx, endnotes, tikz, color, hyperref}

\usepackage{mathtools}

\usepackage[margin=01.5in]{geometry}

\usetikzlibrary{arrows}

\usepackage[nodisplayskipstretch]{setspace}
\def\CC{{\mathbb C}}
\def\FF{{\mathbb F}}

\def\QQ{{\mathbb Q}}
\def\PP{{\mathbb P}}
\def\QQ{{\mathbb Q}}
\def\RR{{\mathbb R}}

\def\ZZ{{\mathbb Z}}

\def\Psf{{\mathsf P}}

\def\hhat{{\hat h}}

\def\0{{\mathbf 0}}
\def\1{{\mathbf 1}}

\def\Acal{{\mathcal A}}
\def\Bcal{{\mathcal B}}

\def\Ocal{{\mathcal O}}

\def\Kbar{{\bar K}}

\def\Qbar{\overline{\QQ}}

\def\Aut{\mathrm{Aut}}

\def\ab{\mathrm{ab}}
\def\Gal{\mathrm{Gal}}

\def\PGL{\mathrm{PGL}}

\def\supp{\mathrm{supp}}

\def\max{\mathrm{max}}

\theoremstyle{plain}

\newtheorem{thm}{Theorem}
\newtheorem{conj}{Conjecture}

\newtheorem{prop}[thm]{Proposition}
\newtheorem{lem}[thm]{Lemma}

\theoremstyle{definition}

\newtheorem{rem}{Remark}

\title[Non-abelian arboreal Galois groups]{\vspace{-2cm} Non-abelian arboreal Galois groups \\ associated to PCF rational maps}

\author{Chifan Leung}

\address{Chifan Leung \newline Department of Mathematics \newline Oregon State University \newline Corvallis, Oregon, U.S.A. \newline leungchi@oregonstate.edu}

\author{Clayton Petsche}

\address{Clayton Petsche \newline Department of Mathematics \newline Oregon State University \newline Corvallis, Oregon, U.S.A. \newline petschec@oregonstate.edu}

\thanks{{\em Date:} July 24, 2024}

\begin{document}

\begin{abstract}
We prove that arboreal Galois extensions of number fields are never abelian for post-critically finite rational maps and non-preperiodic base points.  For polynomials, this establishes a new class of known cases of a conjecture of Andrews-Petsche.  Together with a result of Ferraguti-Ostafe-Zannier, this result implies that counterexamples to the conjecture, if they exist, are sparse.  We also prove an auxiliary result on places of periodic reduction for rational maps, which may be of independent interest.
\end{abstract}

\maketitle


\section{Introduction}\label{IntroSect}

Let $K$ be a number field with algebraic closure $\Kbar$.  Let $f(x)\in K(x)$ be a rational map of degree $d\geq2$, viewed as a dynamical system
\[
f:\PP^1(\Kbar)\to\PP^1(\Kbar).
\]
Denote by $f^n=f\circ\dots\circ f$ the $n$-fold composition of $f$ with itself.  Fix a point $\alpha\in \PP^1(K)$, and for each $n\geq1$, define the $n$-th inverse image set of the pair $(f,\alpha)$ by 
\begin{equation*}
f^{-n}(\alpha) = \{\beta\in\PP^1(\Kbar)\mid f^n(\beta)=\alpha\}.
\end{equation*}
Assume that $\alpha$ is not an exceptional point for $f$; this means that the backward orbit $f^{-\infty}(\alpha)=\cup_{n\geq0}f^{-n}(\alpha)$ of $\alpha$ is an infinite set. 

For each $n\geq1$, let $K_n=K_n(f,\alpha)$ be the field generated over $K$ by $f^{-n}(\alpha)$.  Since $f$ is defined over $K$ and the generators of $K_n$ are $f$-images of generators of $K_{n+1}$, we obtain a tower 
\[
K=K_0\subseteq K_1\subseteq K_2\subseteq\dots
\] 
of Galois extensions of $K$.  Define $K_\infty=K_\infty(f,\alpha)=\cup_{n\geq0}K_n(f,\alpha)$.

The Galois group of the extension $K_\infty/K$ acts faithfully on the infinite rooted tree $T_\infty$ whose vertices are indexed by the points of the infinite backward orbit $f^{-\infty}(\alpha)$.  The resulting injective group homomorphism 
\[
\rho :\Gal(K_\infty/K)\hookrightarrow\Aut(T_\infty)
\]
is known as the arboreal Galois representation associated to the pair $(f,\alpha)$.  This construction is well-studied, going back to Odoni (\cite{MR805714}, \cite{MR813379}, \cite{MR962740}, \cite{MR1418355}) who was motivated by a problem of elementary number theory on prime divisors of recursively defined sequences.  Further work has been done by Stoll \cite{MR1174401}, Boston-Jones \cite{MR2318536}, \cite{MR2520459}, Jones \cite{jones:thesis}, \cite{jones:itgaltow}, \cite{jones:denprimdiv}, \cite{MR3220023}, and others.  Much of the work in this area has been in the direction of showing that, in many cases of interest, the arboreal Galois representation $\rho :\Gal(K_\infty/K)\hookrightarrow\Aut(T_\infty)$ is surjective, or at least that its image has finite index in $\Aut(T_\infty)$; see for example \cite{MR3937585,MR4057534,MR3937588}. 
 
In the opposite direction of large-image results, it would be interesting to know precisely when the image of $\rho :\Gal(K_\infty/K)\hookrightarrow\Aut(T_\infty)$ is as small as possible.  In the stable case, that is when $\Gal(K_n/K)$ acts transitively on $f^{-n}(\alpha)$ for all $n\geq1$, this minimality occurs precisely when $\Gal(K_\infty/K)$ is abelian.  However, the question of characterizing when $\Gal(K_\infty/K)$ is abelian remains interesting even in the absence of a stability hypothesis.  In the special case that $f(x)\in K[x]$ is a polynomial, Andrews-Petsche formulated a conjectural characterization, as follows.

If $f(x),g(x)\in K[x]$ and $\alpha,\beta\in K$, and if $L/K$ is a field extension, we say that the pairs $(f,\alpha)$ and $(g,\beta)$ are $L$-conjugate if $g=\varphi^{-1}\circ f\circ \varphi$ and $\alpha=\varphi(\beta)$ for some affine automorphism $\varphi(x)=ax+b\in L[x]$ with $a\neq0$.

\begin{conj}[Andrews-Petsche \cite{MR4150256}, Ferraguti-Ostafe-Zannier \cite{MR4686319}]\label{PolynomialConjecture}
Let $K$ be a number field, let $f(x)\in K[x]$ be a polynomial of degree $d\geq2$, let $\alpha\in K$, and assume that $\alpha$ is not an exceptional point for $\phi$.  Then $\Gal(K_\infty/K)$ is abelian if and only if the pair $(f,\alpha)$ is $K^\ab$-conjugate to the pair $(g,\beta)$ occurring in one of the following two families of examples:
\begin{itemize}
	\item[(i)] $g(x)=x^{d}$ and $\beta=\zeta$, a root of unity in $\Kbar$.
	\item[(ii)] $g(x)=\pm T_d(x)$ is the $d$-th Chebyshev polynomial and $\beta=\zeta+\zeta^{-1}$, where $\zeta$ is a root of unity in $\Kbar$.
\end{itemize}
\end{conj}

\begin{rem}
In \cite{MR4150256}, Andrews-Petsche incorrectly omitted the case of $-T_d(x)$ among abelian arboreal Galois groups; these maps must be included, because they do give rise to abelian arboreal extensions, but $T_d(x)$ and $-T_d(x)$ are not conjugate to each other when $d$ is odd.  This error was pointed out and corrected by Ferraguti-Ostafe-Zannier \cite{MR4686319}.)
\end{rem}

Conjecture \ref{PolynomialConjecture} remains open, but a number of partial results have been established, including the following.  Recall that a rational map $f(x)\in K(x)$ is said to be {\em postcritically finite}, or {\em PCF}, if every critical point of $f(x)$ in $\PP^1(\Kbar)$ is preperiodic for $f$.
\begin{itemize}
\item[(a)] Conjecture \ref{PolynomialConjecture} has been proved when $f(x)$ is $\Kbar$-conjugate to a powering or Chebyshev map; this is due to Andrews-Petsche \cite{MR4150256}. 
\item[(b)] Conjecture \ref{PolynomialConjecture} has been proved when $K=\QQ$; this is due to results of Ostafe \cite{MR3611309} and Ferraguti-Ostafe-Zannier \cite{MR4686319}, in combination with the Kronecker-Weber theorem. Earlier work of Andrews-Petsche \cite{MR4150256} and Ferraguti-Pagano \cite{MR4216695} had established partial results in this direction.
\item[(c)] Conjecture \ref{PolynomialConjecture} has been proved when $f(x)$ is not PCF; this is due to a result of Ferraguti-Ostafe-Zannier \cite{MR4686319} stating that $\Gal(K_\infty/K)$ is always nonabelian when $f(x)$ is not PCF.
\end{itemize}

According to Conjecture \ref{PolynomialConjecture}, it is expected that $\Gal(K_\infty/K)$ is never abelian when the base point $\alpha$ is not $f$-preperiodic.  The main result of this paper confirms this for all PCF rational maps.

\begin{thm} \label{MainThmIntro}
Let $K$ be a number field, let $f(x)\in K(x)$ be a PCF rational map of degree $d\geq2$, and let $\alpha\in \PP^1(K)$ be a non-preperiodic point for $f$.  Then $\Gal(K_\infty(f,\alpha)/K)$ is not abelian.
\end{thm}

Thus the current state of knowledge of Conjecture \ref{PolynomialConjecture} may be summarized in the following table.

\bigskip

\begin{center}
\begin{tabular}{|c|c|c|}
 \hline
  &   \begin{tabular}{c} $f(x)\in K[x]$ \\ is not PCF \end{tabular} & \begin{tabular}{c} $f(x)\in K[x]$ \\ is PCF \end{tabular}  \\
 \hline
 \begin{tabular}{c}$\alpha\in K$ \\ is not preperiodic \end{tabular}  &  \begin{tabular}{c} Known to be true by \\ Ferraguti et al. \cite{MR4686319} \end{tabular}  & \begin{tabular}{c} Known to be true by \\ Theorem~\ref{MainThmIntro} \end{tabular} \\
 \hline
 \begin{tabular}{c}$\alpha\in K$ \\ is preperiodic \end{tabular}  & \begin{tabular}{c} Known to be true by \\ Ferraguti et al. \cite{MR4686319} \end{tabular} & \begin{tabular}{c} Open in general, \\ but known to be true \\ in certain cases, \\ such as those cited \\ in (a) and (b) above \end{tabular} \\
  \hline
\end{tabular}
\end{center}

\bigskip

It is therefore now the case that in order the finish the proof of Conjecture \ref{PolynomialConjecture}, or to prove any generalization of this conjecture to rational maps (rather than just polynomials), it now suffices to consider only pairs $(f,\alpha)$ in which $f$ is PCF and $\alpha$ is $f$-preperiodic.  

The proof of Theorem~\ref{MainThmIntro} relies on the simple principle that when $\beta\in \Kbar$ is contained in an abelian extension of $K$, then for each place $v$ of $K$, the completion $K_v$ contains either all, or none, of the $\Gal(\Kbar/K)$-conjugates of $\beta$; partial splitting is impossible in the abelian case (Lemma \ref{AbelianSplittingLemma}).  If certain other conditions on the place $v$ are met, this fact can be iterated infinitely many times along the backward orbit of the base point $\alpha$ to obtain a contradiction in the case that $f$ is PCF, $\alpha$ is not $f$-preperiodic, and $\Gal(K_\infty/K)$ is abelian. The proof uses the dynamical equidistribution theorem on Berkovich space due to Baker-Rumely \cite{MR2244226}, Chambert-Loir \cite{MR2244803}, and Favre-Rivera-Letelier \cite{MR2092012}.  The proof also uses a deep result on attracting cycles in non-Archimedean dynamics due to Benedetto-Ingram-Jones-Levy \cite{MR3265554}.

Theorem \ref{MainThmIntro} would be false if the hypothesis that $\alpha$ is not preperiodic were removed, as the examples $(x^d,\zeta)$ and $(\pm T_d(x),\zeta+\frac{1}{\zeta})$ show.  In terms of our proof, this hypothesis is used in several places.  Since $\alpha$ is not periodic, any traversal along the backward orbit of $\alpha$ always produces an infinite sequence of distinct points, which is needed to use the dynamical equidistribution theorem.  Since $\alpha$ is not strictly preperiodic, it is possible to show (Theorem \ref{PeriodicReductionTheorem}) that there exist infinitely many places $v$ at which $\bar\alpha$ is $\bar{f}$-periodic over the residue field $\FF_v$.  Together with the assumption that $f$ is PCF, the hypothesis that $\alpha$ is not preperiodic is also used crucially to ensure that the forward orbit of $\alpha$ does not meet the critical locus of $f$.

Combining the results of Ferraguti-Ostafe-Zannier \cite{MR4686319} with Theorem~\ref{MainThmIntro} of this paper, we can conclude that examples of abelian arboreal Galois groups are sparse.  To make this idea more precise we use the following definition.  

For a polynomial $f(x)\in\Qbar[x]$ and a base point $\alpha\in\Qbar$ denote by $\QQ(f,\alpha)$ the extension of $\QQ$ generated by $\alpha$ and the coefficients of $f(x)$.  We say that $(f,\alpha)\in\Qbar[x]\times\Qbar$ is a {\em potentially abelian arboreal pair} if there exists a finite extension $K/\QQ(f,\alpha)$ for which $\Gal(K_\infty(f,\alpha)/K)$ is abelian.  Thus potentially abelian arboreal pairs include $(x^d,\zeta)$ and $(\pm T_d(x),\zeta+\frac{1}{\zeta})$ for roots of unity $\zeta$.  Moreover, any $\Qbar$-conjugate of a potentially abelian arboreal pair is again potentially abelian.  

\begin{thm}\label{IntroSparsityTheorem}
Let $d\geq2$ be an integer.  For each $T\geq1$, the set of all potentially abelian arboreal pairs $(f,\alpha)\in \Qbar[x]\times\Qbar$ with $\deg(f)=d$ and $[\QQ(f,\alpha):\QQ]\leq T$ is contained in a finite union of $\Qbar$-conjugacy classes in $\Qbar[x]\times\Qbar$.
\end{thm}

Since Conjecture \ref{PolynomialConjecture} is known in all cases where $\Gal(K_\infty(f,\alpha)/K)$ is expected to be abelian, we may interpret Theorem \ref{IntroSparsityTheorem} as implying that counterexamples to Conjecture \ref{PolynomialConjecture}, if they exist, are sparse.

The authors would like to thank Tom Tucker for helpful discussions.

\section{The dynamics of rational maps with good reduction}

Let $K$ be a number field, and let $v$ be a non-Archimedean place of $K$.  Fix any choice of normalized absolute value $|\cdot|_v$ associated to $v$.  Let 
\[
\Ocal_v=\{x\in K_v\mid|x|_v\leq1\}
\] 
be the ring of $v$-adic integers in $K_v$, and let $\pi_v\in\Ocal_v$ be a uniformizer; thus $\pi_v\Ocal_v$ is the unique maximal ideal of $\Ocal_v$.  Let $\FF_v=\Ocal_v/\pi_v\Ocal_v$ be the residue field, and let $x\mapsto\bar x$ denote the reduction map $\Ocal_v\to\FF_v$.  Denote by $\epsilon_v=|\pi_v|_v$ the absolute value of the uniformizer.  Thus $|K_v^\times|_v=\epsilon_v^\ZZ$ is the value group of the local field $K_v$.

Using homogeneous coordinates $(x_1:x_2)$ on $\PP^1(K_v)$, we can identify $\PP^1(K_v)$ with $K_v\cup\{\infty\}$ by the correspondence $(\alpha:1)=\alpha$ for $\alpha\in K_v$, and $(1:0)=\infty$.  The reduction map $\Ocal_v\to\FF_v$ can be extended to a map $\PP^1(K_v)\to \PP^1(\FF_v)$ on projective spaces in which, for each $x\in \PP^1(K_v)$, we select homogenous coordinates $x=(x_1,x_2)$ with $\max(|x_1|_v,|x_2|_v)=1$, and we define $\bar{x}=(\bar{x}_1,\bar{x}_2)$ in $\PP^1(\FF_v)$.

Recall that the projective metric $\delta_v(\cdot,\cdot)$ on $\PP^1(K_v)$ is defined by
\begin{equation*}
\delta_v(x,y) = \frac{|x_1y_2-x_2y_1|_v}{\max(|x_1|_v,|x_2|_v)\max(|y_1|_v,|y_2|_v)},
\end{equation*}
where $x=(x_1:x_2)$ and $y=(y_1:y_2)$ are points in $\PP^1(K_v)$.  It is well known that this defines a metric on $\PP^1(K_v)$, and that given two points $x,y\in\PP^1(K_v)$, we have $0\leq \delta_v(x,y)\leq 1$, with $\delta_v(x,y)<1$ if and only if $\bar{x}=\bar{y}$ in $\PP^1(\FF_v)$.  For proofs of these facts see \cite{MR2316407} $\S$ 2.1-2.3.  Note also that when $x,y\in\Ocal_v$, we have $\delta_v(x,y)=\delta_v((x:1),(y:1))=|x-y|_v$, and thus the projective metric coincides with the ordinary $v$-adic metric on $v$-integral points of $K_v$.

For each choice of center $c\in\PP^1(K_v)$ and radius $r\in \epsilon_v^\ZZ$, we denote the disc of radius $r$ about $c$ by $D_r(c) = \{x\in\PP^1(K_v)\mid \delta_v(x,c)\leq r\}$.  Letting $q_v=|\FF_v|$ denote the size of the residue field, it follows from the observations in the preceding paragraph that $\PP^1(K_v)$ is a disjoint union of $q_v+1$ discs of radius $\epsilon_v$, and that these discs are precisely the fibers of the reduction map $\PP^1(K_v)\to \PP^1(\FF_v)$.

Let $f(x)\in K_v(x)$ be a rational map of degree $d\geq2$.  We declare that $f$ has {\em good reduction} if $f(x)=p(x)/q(x)$ for polynomials $p(x),q(x)\in\Ocal_v[x]$ such that: 
\begin{itemize}
\item[(i)] $\max(\deg p,\deg q)=d$; 
\item[(ii)] $\max(\deg \bar{p},\deg \bar{q})=d$, where $\bar{p}(x),\bar{q}(x)\in\FF_v[x]$ are the polynomials obtained by reducing modulo $\pi_v$ the coefficients of $p(x)$ and $q(x)$; and 
\item[(iii)] $\bar{p}(x)$ and $\bar{q}(x)$ have no common roots in $\overline{\FF}_v$.  
\end{itemize}

In particular, when $f$ has good reduction at $v$, it induces a well-defined rational map $\bar{f}(x)\in\FF_v(x)$ of degree $d$ defined over the residue field $\FF_v$.  Moreover, such $f(x)$ is non-expanding in the sense that $\delta_v(f(x),f(y))\leq\delta_v(x,y)$ for all $x,y\in\PP^1(K_v)$, and the diagram
\begin{equation*}
\begin{CD}
\PP^1(K_v)  @> f >>   \PP^1(K_v) \\ 
@V \text{reduction} VV                                    @VV \text{reduction} V \\ 
\PP^1(\FF_v)           @> \bar{f} >>     \PP^1(\FF_v)
\end{CD} 
\end{equation*}
commutes; i.e. $\overline{f(x)}=\bar{f}(\bar{x})$ in $\PP^1(\FF_v)$ for all $x\in\PP^1(K)$.  For proofs of these facts, see \cite{MR2316407} Theorems 2.17 and 2.18.

For a rational map with good reduction, the following Lemma describes the local dynamics at a fixed point over the residue field.

\begin{lem} \label{MainLocalLemma}
Let $K$ be a number field and let $v$ be a non-Archimedean place of $K$. Let $f(x)\in K_v(x)$ be a rational map of degree $d\geq2$ with good reduction, and let $\alpha\in K_v$ such that $|\alpha|_v\leq1$ and $\bar{f}(\bar{\alpha})=\bar\alpha$ in $\PP^1(\FF_v)$.
\begin{itemize}
\item[{\bf (a)}] The map $f$ restricts to a map $f: D_{\epsilon_v}(\alpha)\to D_{\epsilon_v}(\alpha)$.
\item[{\bf (b)}] For all $\beta\in D_{\epsilon_v}(\alpha)$, it holds that $|f'(\beta)|_v\leq1$.
\item[{\bf (c)}] If $|f'(\alpha)|_v<1$ then $|f'(\beta)|_v\leq\epsilon_v$ for all $\beta\in D_{\epsilon_v}(\alpha)$, and $D_{\epsilon_v}(\alpha)$ contains an attracting fixed point of $f$.
\item[{\bf (d)}] If $|f'(\alpha)|_v=1$ then $f: D_{\epsilon_v}(\alpha)\to D_{\epsilon_v}(\alpha)$ is surjective.  
\end{itemize}
\end{lem}
\begin{proof}$\,$

{\bf (a)}  The assumption $\bar{f}(\bar{\alpha})=\bar\alpha$ means that $\delta_v(f(\alpha),\alpha)<1$, and so actually that $\delta_v(f(\alpha),\alpha)\leq\epsilon_v$; in other words $f(\alpha)\in D_{\epsilon_v}(\alpha)$.  Since any point in a non-Archimedean disc can be the center, we then have $D_{\epsilon_v}(f(\alpha))=D_{\epsilon_v}(\alpha)$.  Given any $\beta\in D_{\epsilon_v}(\alpha)$, the non-expanding property of maps with good reduction (\cite{MR2316407} Theorem 2.17) implies that
\begin{equation*}
\delta_v(f(\alpha),f(\beta)) \leq \delta_v(\alpha,\beta) \leq\epsilon_v
\end{equation*}
and hence $f(D_{\epsilon_v}(\alpha))\subseteq D_{\epsilon_v}(f(\alpha))=D_{\epsilon_v}(\alpha)$.

{\bf (b)}  Write $f(x)=p(x)/q(x)$ for $p(x),q(x)\in\Ocal_v[x]$ as in the definition of good reduction.  For $\beta\in D_{\epsilon_v}(\alpha)$ we have 
\begin{equation}\label{xInDiscEstimates}
\begin{split}
|q(\beta)|_v & = 1 \\
|f'(\beta)|_v & = |q(\beta)p'(\beta)-p(\beta)q'(\beta)|_v\leq1.
\end{split}
\end{equation}
Indeed, $\bar{\beta}=\bar{\alpha}$ and so $\bar{f}(\bar{\beta})=\bar{f}(\bar{\alpha})=\bar\alpha\neq\infty$ in $\PP^1(\FF_v)$, we deduce that $\bar{q}(\bar{\beta})\neq0$ in $\FF_v$.  The first identity in (\ref{xInDiscEstimates}) follows immediately, and from this the second part of (\ref{xInDiscEstimates}) follows from the quotient rule and the strong triangle inequality.  

{\bf (c)}  From $|f'(\alpha)|_v<1$, it follows from the quotient rule and $|q(\alpha)|_v=1$ that $\bar{f}'(\bar\alpha)=0$ in $\FF_v$.  Given any $\beta\in D_{\epsilon_v}(\alpha)$ we have $\bar{\beta}=\bar\alpha$ and hence $\bar{f}'(\bar{\beta})=\bar{f}'(\bar{\alpha})=0$ in $\FF_v$, so $|f'(\beta)|_v<1$.  In particular, $|f'(\beta)|_v\leq \epsilon_v$.

Fix arbitrary $\beta,\gamma\in D_{\epsilon_v}(\alpha)$.  Then expanding in a power series about $\gamma$ we have
\begin{equation}
\begin{split}
p(x) & = p(\gamma) + p'(\gamma)(x-\gamma) + a_2(x-\gamma)^2+\dots \\
q(x) & = q(\gamma) + q'(\gamma)(x-\gamma) + b_2(x-\gamma)^2+\dots \\
\end{split}
\end{equation}
for some $a_j,b_j\in\Ocal_v$ depending on $\gamma$.  That the $a_j$ and $b_j$ are in $\Ocal_v$ follows from the fact that the $a_j$ are in fact the coefficients of $P(x)=p(x+\gamma)$ which is in $\Ocal_v[x]$ because both $\gamma$ and the coefficients of $p(x)$ are in $\Ocal_v$; similarly for the $b_j$.

Substituting $x=\beta$ in each of the above expansions of $p(x)$ and $q(x)$ and simplifying we have
\begin{equation*}
\begin{split}
f(\beta)-f(\gamma)& = \frac{p(\beta)q(\gamma)-p(\gamma)q(\beta)}{q(\beta)q(\gamma)} \\
	& = \frac{(p'(\gamma)q(\gamma)-p(\gamma)q'(\gamma))(\beta-\gamma)+\theta(\beta-\gamma)^2}{q(\beta)q(\gamma)} \\
\end{split}
\end{equation*}
where $|\theta|_v\leq1$.  Using that $|q(\gamma)|_v=|q(\beta)|_v=1$, part {\bf (b)} of this lemma, and that $|f'(\gamma)|_v=|p'(\gamma)q(\gamma)-p(\gamma)q'(\gamma)|_v\leq\epsilon_v$ and $|\beta-\gamma|_v\leq \epsilon_v$ we obtain 
\begin{equation*}
|f(\beta)-f(\gamma)|_v \leq \epsilon_v|\beta-\gamma|_v.
\end{equation*}
We conclude that the map $f: D_{\epsilon_v}(\alpha)\to D_{\epsilon_v}(\alpha)$ is a contraction, and the Banach fixed-point theorem provides a unique attracting fixed point of $f$ in $D_{\epsilon_v}(\alpha)$.

{\bf (d)}  Consider $\beta\in D_{\epsilon_v}(\alpha)$, and we seek $\gamma\in D_{\epsilon_v}(\alpha)$ for which $f(\gamma)=\beta$.  Such a $\gamma$ would be a root of $F(x)=p(x)-\beta q(x)$ in $D_{\epsilon_v}(\alpha)$, and would exist by Hensel's Lemma, as long as we can check the hypotheses that
\begin{equation}\label{HenselConditions}
\begin{split}
|F(\alpha)|_v & <1\text{, and } \\
|F'(\alpha)|_v & =1.
\end{split}
\end{equation}

Since $\bar{f}(\bar\alpha)=\bar\alpha=\bar\beta$, we have $f(\alpha)=\frac{p(\alpha)}{q(\alpha)}=\beta+\theta$ for $\theta\in K_v$ with $|\theta|_v<1$.  Therefore
\[
|F(\alpha)|_v = |p(\alpha)-\beta q(\alpha)|_v = |\theta q(\alpha)|_v < 1
\]
as $|q(\alpha)|_v=1$; this completes the first part of (\ref{HenselConditions}).  We also have
\begin{equation*}
\begin{split}
F'(\alpha) & = p'(\alpha)-\beta q'(\alpha) \\
	& = p'(\alpha)-\left(\frac{p(\alpha)}{q(\alpha)}-\theta\right)q'(\alpha) \\
	& = \frac{p'(\alpha)q(\alpha)-p(\alpha)q'(\alpha)}{q(\alpha)} +\theta q'(\alpha) \\
\end{split}
\end{equation*}
By hypothesis, we have $|f'(\alpha)|_v=|p'(\alpha)q(\alpha)-p(\alpha)q'(\alpha)|_v=1$, and this together with $|q(\alpha)|_v=1$ and $|\theta q'(\alpha)|_v\leq|\theta|_v<1$ shows that $|F'(\alpha)|_v=1$ by the case of equality in the strong triangle inequality, completing the proof of (\ref{HenselConditions}) and of the Lemma.
\end{proof}

\section{A result on periodic reduction}

The following result is needed in the proof of Theorem \ref{MainThmIntro}, and may also be of independent interest.  Given a rational map $f(x)\in K(x)$ and a point $\alpha\in \PP^1(K)$, we say that $\alpha$ is {\em strictly preperiodic} with respect to $f$ if $\alpha$ is preperiodic but not periodic.

\begin{thm}\label{PeriodicReductionTheorem}
Let $K$ be a number field, let $f(x)\in K(x)$ be a rational map of degree $d\geq2$, and let $\alpha\in \PP^1(K)$ be a point which is not strictly preperiodic for $f$.  Then there exist infinitely many non-Archimedean places $v$ of $K$ at which $f$ has good reduction and at which $\bar\alpha$ is $\bar{f}$-periodic in $\PP^1(\FF_v)$.
\end{thm}

Benedetto-Ghioca-Kurlberg-Tucker \cite{MR2885573} have proved a complementary result, that when $\alpha\in \PP^1(K)$ is {\em not $f$-periodic}, then there exist infinitely many non-Archimedean places $v$ of $K$ at which $f$ has good reduction and at which $\bar\alpha$ is {\em not $\bar{f}$-periodic} in $\PP^1(\FF_v)$. 

We offer two proofs of Theorem \ref{PeriodicReductionTheorem}, both of which rely ultimately on well-known results in Diophantine approximation.  

Let $S$ be a finite set of places of $K$ including all of the Archimedean places.  Given distinct points $x,y\in\PP^1(K)$, we say that the pair $\{x,y\}$ is {\em $S$-integral} if $\bar{x}\neq\bar{y}$ in $\PP^1(\FF_v)$ for all places $v$ of $K$ outside $S$.  If $X$ is any subset of $\PP^1(K)$, we say that the set $X$ is {\em $S$-integral} if, for all pairs $x,y$ of distinct points in $X$, the pair $\{x,y\}$ is $S$-integral.

Thus to say that the set $\{\infty, x\}$ is $S$-integral is the same thing as saying that $|x|_v\leq1$ for all $v\notin S$, which is also equivalent to the statement that $x$ is an element of the ring $\Ocal_S$ of $S$-units in $K$.  To say the set $\{0,\infty,x\}$ is $S$-integral is the same as saying that $x$ is an $S$-unit, in other words that $|x|_v=1$ for all places $v\notin S$.

\begin{lem}\label{FiniteSLemma}
Let $K$ be a number field and let $X$ be a finite subset of $\PP^1(K)$.  Then there exists a finite set $S$ of places of $K$ such that $X$ is $S$-integral.
\end{lem}
\begin{proof}
In the special case $X=\{\alpha,\infty\}$ for $\alpha\in K$, this is just the well known fact that $|\alpha|_v\leq1$ for all but finitely many places $v$.  If $X=\{\alpha,\beta\}$ for distinct $\alpha,\beta\in K$, then $X$ is $S$-integral provided $S$ contains all of the finitely many non-Archimedean places for which either $|\alpha|_v>1$, $|\beta|_v>1$, or $|\alpha-\beta|_v\neq1$.  Finally, for arbitrary finite $X=\{\alpha_1,\alpha_2,\dots,\alpha_n\}$, the result follows by taking $S$ to be the union of the finite sets of places with respect to which each of the pairs $\{\alpha_i,\alpha_j\}$ is $S$-integral.
\end{proof}

\begin{lem}\label{S-IntegralLemma}
Let $K$ be a number field, and let $S$ be a finite set of places of $K$ including all of the Archimedean places.  Suppose that $\{\alpha_1,\alpha_2,\alpha_3\}$ is an $S$-integral set of three distinct points in $\PP^1(K)$.  Then there exist only finitely many $\gamma\in \PP^1(K)$ such that $\{\alpha_1,\alpha_2,\alpha_3,\gamma\}$ is $S$-integral.
\end{lem}
\begin{proof}
This is a well-known variation on a standard finiteness result on $S$-units.  To sketch the proof, after applying a M\"obius transformation and possibly enlarging $S$, it suffices to show that there are only finitely many $\gamma\in K$ such that $\{0,1,\infty,\gamma\}$ is $S$-integral.  Each such $\gamma$ leads to a solution $(\gamma,1-\gamma)$ in $S$-units to the equation $x+y=1$, of which it is well known that there can only be finitely many; see for example \cite{bombierigubler}, Theorem 5.2.1.
\end{proof}

\begin{proof}[Proof of Theorem \ref{PeriodicReductionTheorem}]
Since $\alpha$ is not strictly preperiodic with respect to $f$, either $\alpha$ is periodic or it is not preperiodic at all.  If $\alpha$ is $f$-periodic, then $\bar\alpha$ is $\bar{f}$-periodic in $\PP^1(\FF_v)$ whenever $v$ is a place at which $v$ has good reduction, so the result is trivial in this case.  

Hence we may now assume that $\alpha$ is not $f$-preperiodic.  Assume that the desired conclusion of the Theorem is false.  Thus there exists a finite set of places $S$ of $K$ such that, for all places $v\not\in S$, $f$ has good reduction at $v$ and $\bar\alpha$ is not $\bar{f}$-periodic in $\PP^1(\FF_v)$.

For each $m\geq0$ set $\alpha_m=f^m(\alpha)$.  Extend the sequence $\{\alpha_m\}$ by selecting $\alpha_{-2},\alpha_{-1}\in\PP^1(\Kbar)$ with $f(\alpha_{-1})=\alpha_0=\alpha$ and $f(\alpha_{-2})=\alpha_{-1}$.  Let $L=K(\alpha_{-2},\alpha_{-1})$.  Let $T$ be the set of all places of $L$ lying above places in $S$. 

Let $w$ be a place of $L$ with $w\notin T$ and let $v$ be the place of $K$ lying under $w$; thus $f(x)$ still has good reduction at $w$ because it has good reduction at $v$.  Moreover, $\bar\alpha_0=\bar\alpha$ is not periodic in $\PP^1(\FF_w)$, because $\bar{f}^m(\bar\alpha)=\bar\alpha$ in $\PP^1(\FF_w)$ is also valid as an identity in $\PP^1(\FF_v)$, as both $f$ and $\alpha$ are defined over $K$.  

Note that for each individual $m\geq1$ the four points $\bar\alpha_{-2},\bar\alpha_{-1},\bar\alpha_{0},\bar\alpha_{m}$ are distinct in $\PP^1(\FF_w)$.  Indeed, $\bar\alpha_{0}\neq\bar\alpha_{m}$ in $\PP^1(\FF_w)$ because $\bar\alpha_{0}$ is not $\bar{f}$-periodic in $\PP^1(\FF_w)$.  If $\bar\alpha_{-1}=\bar\alpha_{0}$ in $\PP^1(\FF_w)$, then applying $\bar{f}$ to both sides would yield $\bar\alpha_{0}=\bar\alpha_{1}$ in $\PP^1(\FF_w)$, a contradiction because $\bar\alpha_0$ is not $\bar{f}$-periodic in $\PP^1(\FF_w)$; thus $\bar\alpha_{-1}\neq\bar\alpha_{0}$ in $\PP^1(\FF_w)$.  Similar reasoning shows that $\bar\alpha_{-1}\neq\bar\alpha_{m}$, $\bar\alpha_{-2}\neq\bar\alpha_{-1}$, $\bar\alpha_{-2}\neq\bar\alpha_{0}$, and $\bar\alpha_{-2}\neq\bar\alpha_{m}$ in $\PP^1(\FF_w)$.

As $w\not\in T$ was arbitrary, we have obtained a sequence of distinct points $\alpha_m=f^m(\alpha)$ in $\PP^1(L)$ such that the set of four points $\{\alpha_{-2},\alpha_{-1},\alpha_{0},\alpha_{m}\}$ is $T$-integral; this is a violation of Theorem \ref{S-IntegralLemma}, and the contradiction completes the proof.
\end{proof}

The following alternate proof of Theorem \ref{PeriodicReductionTheorem} is slightly simpler than our first proof, but it is less self-contained in that it relies on Silverman's theorem on integral points in orbits.  This proof was shown to us by Tom Tucker, and it is similar in spirit to the argument used in \cite{MR2885573} Lemma 4.3.

\begin{proof}[Alternate proof of Theorem \ref{PeriodicReductionTheorem}]
Again we may assume that $\alpha$ is not $f$-preperiodic as the result is trivial otherwise.  If the desired conclusion is false, then there exists a finite set of places $S$ of $K$ such that, for all places $v\not\in S$, $f$ has good reduction at $v$ and $\bar\alpha$ is not $\bar{f}$-periodic in $\PP^1(\FF_v)$.

Given any point $\beta\in\PP^1(K)$ which is not an exceptional point for $f$, Silverman's integral points in orbits theorem (\cite{MR1240603} Theorem 2.2) implies that there exist at most finitely many $n\geq1$ such that the pair $\{f^n(\alpha),\beta\}$ is $S$-integral.  (The result of \cite{MR1240603} is stated in the special case that $\beta=\infty$, but the more general statement follows trivially from a change of coordinates.) Thus there exists some $n\geq1$ and some place $v\notin S$ for which $\bar{f}^{n}(\bar\alpha)=\bar\beta$ in $\PP^1(\FF_{v})$.  Taking $\beta=\alpha$ implies that $\bar\alpha$ is $\bar{f}$-periodic for some place $v\notin S$, a contradiction of the initial assumption.
\end{proof}

\section{The proof of Theorem \ref{MainThmIntro}}

The main idea behind the proof of Theorem \ref{MainThmIntro} makes essential use of the following simple lemma.

\begin{lem} \label{AbelianSplittingLemma}
Let $K$ be a number field and let $F(x)\in K[x]$ be irreducible with abelian Galois group over $K$.  If $v$ is a place of $K$ and $F(x)$ has a root in $K_v$, then $F(x)$ splits completely over $K_v$.
\end{lem}
\begin{proof}
Suppose that $F(x)$ has a root $\alpha$ in $K_v$, and let $L\subseteq\overline K_v$ be a splitting field for $F(x)$ over $K$.  Since $\Gal(L/K)$ is abelian, the extension $K(\alpha)/K$ is Galois, and because $F(x)$ is irreducible, this means that $F(x)$ splits completely over $K(\alpha)$.  As $K(\alpha)\subseteq K_v$, we obtain that $F(x)$ splits completely over $K_v$.
\end{proof}

Next, we briefly review the notions of the canonical height associated to a rational map, as well as the local equidistriubution of dynamically small points.  Let $f(x)\in K(x)$ be a rational map of degree at least $2$.  The Call-Silverman canonical height $\hhat_f:\PP^1(\Kbar)\to\RR$ is the unique real valued function on $\PP^1(\Kbar)$ which differs from the ordinary absolute Weil height $h:\PP^1(\Kbar)\to\RR$ by a constant bound, and which satisfies the dynamical relation 
\begin{equation}\label{CanonicalHeightFE}
\hhat_f(f(\alpha))=\deg(f)\hhat_f(\alpha)
\end{equation}
for all $\alpha\in\PP^1(\Kbar)$.  See Silverman \cite{MR2316407} $\S$ 3.4 for more background.

For each place $v$ of $K$, let $\Psf_v^1$ be the projective line over $\CC_v$ in the sense of Berkovich. In the archimedean case, $\CC_v=\CC$ and $\Psf_v^1=\PP^1(\CC_v)$ is the usual Riemann sphere.  In the non-Archimedean case, $\Psf_v^1$ is a strictly larger analytic compactification of the ordinary projective line $\PP^1(\CC_v)$; see Berkovich \cite{MR1070709} and Baker-Rumely \cite{MR2599526} for more background on this space.

Given a point $\alpha\in\PP^1(\Kbar)$ and a place $v$ of $K$, let $[\alpha]_v$ be the unit probability measure on $\Psf_v^1$ which is supported equally on each one of the $[K(\alpha):K]$ distinct embeddings of $\alpha$ into $\Psf_v^1$.  In other words,
\[
[\alpha]_v = \frac{1}{[K(\alpha):K]}\sum_{\sigma:K(\alpha)\hookrightarrow\CC_v}\delta_{\sigma(\alpha)},
\]
where $\delta_t$ denotes the Dirac measure on $\Psf_v^1$ supported at the point $t$.  Each embedding $\sigma:K(\alpha)\hookrightarrow\CC_v$ induces an extended embedding $\sigma:\PP^1(K(\alpha))\hookrightarrow\Psf_v^1$ using projective coordinates and the inclusion $\PP^1(\CC_v)\hookrightarrow\Psf_v^1$.

The equidistribution theorem of Baker-Rumely \cite{MR2244226}, Chambert-Loir \cite{MR2244803}, and Favre-Rivera-Letelier \cite{MR2092012} states that, given any rational map $f(x)\in K(x)$ of degree at least 2 and any place $v$ of $K$, there exists a unit Borel measure $\mu_{f,v}$ on $\Psf_v^1$ which describes the limiting distribution of $\Gal(\Kbar/K)$-orbits of dynamically small points for $f$.  The precise meaning of this last statement is that, if $\{\alpha_n\}$ is a sequence of distinct points in $\PP^1(\Kbar)$ with $\hhat_f(\alpha_n)\to0$, then $[\alpha]_v\to\mu_{f,v}$ weakly, in the sense that $\int \varphi d[\alpha_n]_v\to \int\varphi d\mu_{f,v}$ for all continuous functions $\varphi:\Psf_v^1\to\RR$.

The following result, Proposition \ref{MainNotAbelianProp}, is the main technical step of the proof of Theorem \ref{MainThmIntro}.  The idea of the proof is that, if $\Gal(K_\infty/K)$ is abelian and if there exists a non-Archimedean place of $K$ satisfying certain favorable conditions, then Lemma \ref{AbelianSplittingLemma} can be iterated infinitely many times, along a sequence traversing up the tree of iterated inverse images of the base point $\alpha$, leading to a contradiction of the equidistrubtion theorem.  The proof also makes essential use of a deep result of non-Archimedean dynamics due to Benedetto-Ingram-Jones-Levy \cite{MR3265554}.

\begin{prop} \label{MainNotAbelianProp}
Let $f(x)\in K(x)$ be a PCF rational map of degree $d\geq2$ and let $\alpha\in K$ be a non-preperiodic point for $f$.  Let $K_\infty=K_\infty(f,\alpha)$ be the infinite arboreal Galois extension of $K$ associated to the pair $(f,\alpha)$, as described in $\S$ \ref{IntroSect}.  Suppose that there exists a non-Archimedean place $v$ of $K$ such that all of the following conditions hold.
\begin{itemize}
\item[{\bf (A)}] The characteristic of $\FF_v$ is $>d$,
\item[{\bf (B)}] $f$ has good reduction at $v$,
\item[{\bf (C)}] $|\alpha|_v\leq 1$, and
\item[{\bf (D)}] $\bar\alpha$ is $\bar{f}$-periodic in $\PP^1(\FF_v)$.
\end{itemize}
If $f$ has any periodic critical points in $\PP^1(\Kbar)$, denote by $M$ the largest $f$-period among all periodic critical points of $f$.  If such point(s) exist, assume that the place $v$ also satisfies the following condition. 
\begin{itemize}
\item[{\bf (E)}] $\bar\beta\neq\bar c$ in $\PP^1(\FF_v)$, for all $\beta$ in the first $M$ terms $\alpha,f(\alpha),\dots, f^{M-1}(\alpha)$ of the forward orbit of $\alpha$, and all critical points $c$ of $f$ lying in $\PP^1(K_v)$.
\end{itemize}
Then $\Gal(K_\infty/K)$ is not abelian.
\end{prop}

\begin{proof}
By hypothesis {\bf (D)}, there exists $n\geq1$ such that $\bar{f}^n(\bar\alpha)=\bar\alpha$ in $\FF_v$.  Since $f$ has good reduction at $v$, so does $f^n$ (\cite{MR2316407} Theorem 2.18).  

\medskip

{\bf Step 1:} We first show that
\begin{equation}\label{DerivativeNonvanishingStep}
|(f^n)'(\alpha)|_v=1.
\end{equation}
That $|(f^n)'(\alpha)|_v\leq1$ follows from Lemma \ref{MainLocalLemma} {\bf(b)}.  Assuming that $|(f^n)'(\alpha)|_v<1$, we will obtain a contradiction.  

First, applying Lemma \ref{MainLocalLemma} {\bf (c)}, the assumption $|(f^n)'(\alpha)|_v<1$ implies that there exists an $n$-periodic point $\gamma$ in $D_{\epsilon_v}(\alpha)$ such that $|(f^n)'(\gamma)|_v<1$; in particular, $\gamma\in\PP^1(K_v)$.  If $(f^n)'(\gamma)\neq0$, then since the residue characteristic of $v$ is $>d$, it follows that the cycle containing $\gamma$ strictly attracts a critical orbit by Theorem 1.5 of Benedetto-Ingram-Jones-Levy \cite{MR3265554}, a contradiction of the assumption that $f$ is PCF.  

Thus we may assume that $(f^n)'(\gamma)=0$; by the chain rule this implies that one of the points $\gamma,f(\gamma),f^2(\gamma),\dots,f^{n-1}(\gamma)$ in the cycle containing $\gamma$ is a periodic critical point lying in $\PP^1(K_v)$.  Letting $m$ denote the minimal period of this critical point, it follows from the definition of the integer $M$ that $m\leq M$.  Thus $\gamma$ has minimal period $m$ as well, so one of the points $\gamma,f(\gamma),f^2(\gamma),\dots,f^{M-1}(\gamma)$ is a critical point, say $c=f^i(\gamma)$ is a critical point for $0\leq i\leq M-1$.  It follows that in $\PP^1(\FF_v)$, we have $\bar{c}=\bar{f}^i(\bar\gamma)=\bar{f}^i(\bar\alpha)$, which contradicts condition {\bf (E)} of the hypotheses. The contradiction completes the proof of (\ref{DerivativeNonvanishingStep}).

\medskip

{\bf Step 2:} We turn now to the proof that $\Gal(K_\infty/K)$ is not abelian.  Assuming that $\Gal(K_\infty/K)$ is abelian, we will obtain a contradiction.  Lemma \ref{MainLocalLemma} {\bf (d)} and (\ref{DerivativeNonvanishingStep}) show that $f^n$ restricts to a surjective map 
\[
f^n: D_{\epsilon_v}(\alpha)\to D_{\epsilon_v}(\alpha).
\]
We may therefore define a sequence $\{\alpha_m\}$ in $D_{\epsilon_v}(\alpha)$ recursively by $\alpha_0=\alpha$ and $f^n(\alpha_{m+1})=\alpha_m$; in particular $\alpha_m\in K_v$ for all $m$.  Moreover, the points $\alpha_m$ are distinct, because if $\alpha_{m'}=\alpha_{m}$ for $m'<m$, then $\alpha_m$ would be periodic.  Then $\alpha=\alpha_0=f^{mn}(\alpha_m)$ would be periodic, a contradiction.

Let $F_m(x)$ be the minimal polynomial over $K$ of $\alpha_m$.  The splitting field of $F_m(x)$ is contained in $K_\infty$ and hence is abelian over $K$; therefore applying Lemma \ref{AbelianSplittingLemma}, the polynomial $F_m(x)$ splits completely over $K_v$.  It follow that each $m\geq0$, the measure $[\alpha_m]_v$ is supported on the compact set $\PP^1(K_v)$ of $K_v$-rational type I points of the Berkovich projective line $\Psf_v^1$.  By the equidistribution theorem, we have $[\alpha_m]_v\to\mu_{f,v}$ weakly as $m\to+\infty$, and therefore
\begin{equation}\label{WrongSupport}
\supp(\mu_{f,v})\subseteq\PP^1(K_v).
\end{equation}
But (\ref{WrongSupport}) is false by the good reduction hypothesis on $v$, which implies (\cite{MR2244226}, Example 3.24) that $\mu_{f,v}$ is supported at the Gauss  point $\zeta_{0,1}$ of $\Psf_v^1$, which is not contained in $\PP^1(K_v)$.  The contradiction proves that $\Gal(K_\infty/K)$ is not abelian.
\end{proof}

\begin{proof}[Proof of Theorem \ref{MainThmIntro}]
Let $K$ be a number field, let $f(x)\in K(x)$ be a PCF rational map of degree $d\geq2$, and let $\alpha\in \PP^1(K)$ be a non-preperiodic point for $f$.  Let $K_\infty=K_\infty(f,\alpha)$ be the infinite arboreal Galois extension of $K$ associated to the pair $(f,\alpha)$.  Without loss of generality we may assume that $\alpha\neq\infty$, because if $\alpha=\infty$ then we may replace $f(x)$ with $1/f(1/x)$ and replace $\alpha=\infty$ with $\alpha=0$.

To show that $\Gal(K_\infty/K)$ is not abelian, we only need to show that a non-Archimedean place satisfying the conditions of Proposition \ref{MainNotAbelianProp} can always be found.  It is clear that conditions {\bf (A)}, {\bf (B)}, and {\bf (C)} hold at all except finitely many places of $K$.  

Condition {\bf (E)} does not need to be checked if $f$ has no periodic critical points in $\PP^1(\Kbar)$.  Assuming that $f$ does have at least one periodic critical point in $\PP^1(\Kbar)$, we can show that condition {\bf (E)} also holds at all except finitely many places.  Let $L/K$ be a finite extension with the property that all critical points of $f$ are $L$-rational.  It follows from Lemma \ref{FiniteSLemma} that there exists a finite set of places $T$ of $L$ such that the subset
\[
\{\text{critical points of }f\}\cup\{\alpha,f(\alpha),\dots, f^{M-1}(\alpha)\}
\]
of $\PP^1(L)$ is $T$-integral.  Observe also that
\[
\{\text{critical points of }f\}\cap\{\alpha,f(\alpha),\dots, f^{M-1}(\alpha)\}=\emptyset
\]
because $f$ is post critically finite and hence its critical locus does not meet the orbit of the non-preperiodic point $\alpha$.  Letting $S$ denote the set of all places of $K$ lying below any of the places in $T$, we have now demonstrated that  condition {\bf (E)} holds for all places $v$ of $K$ such that $v\notin S$.

Finally, condition {\bf (D)} holds at infinitely many places of $K$ by Theorem \ref{PeriodicReductionTheorem}, completing the proof.
\end{proof}

\section{Sparsity of abelian arboreal Galois groups}

In this section we prove Theorem \ref{IntroSparsityTheorem}.  

\begin{lem}\label{DegreeBoundLemma}
Let $d\geq2$ be an integer, let $K$ be a field, and let $f(x),g(x)\in K(x)$ be rational maps of degree $d$ such that $g=\varphi^{-1}\circ f\circ \varphi$ for some automorphism $\varphi\in\PGL_2(\Kbar)$.  Then $\varphi$ is defined over a field extension $L/K$ of degree $[L:K]\leq d^4(d+1)^2$.  

If $f$ and $g$ are polynomials and $\varphi(x)=ax+b\in\Kbar[x]$ is an affine automorphism, then the stronger bound $[L:K]\leq d^2(d+1)^2$ holds.
\end{lem}

\begin{proof}
Let $\alpha_1\in\PP^1(\Kbar)$ be a fixed point of $f$ which is not totally ramified; such a point always exists by an elementary argument, or one could use \cite{MR2316407} Theorem 1.14.  Since $\alpha_1$ is not totally ramified we can find a point $\alpha_2\in f^{-1}(\alpha_1)$ which is distinct from $\alpha_1$.  Selecting any $\alpha_3\in f^{-1}(\alpha_2)$, we now have a triple $\{\alpha_1,\alpha_2,\alpha_3\}$ of distinct points in $\PP^1(\Kbar)$ with directed graph 
\begin{center}
\begin{tikzpicture}[node distance={15mm}, thick] 
\node (1) {$\alpha_3$}; 
\node (2) [right of=1] {$\alpha_2$};
\node (3) [right of=2] {$\alpha_1$};
\draw[->] (1) -- (2);
\draw[->] (2) -- (3);
\draw[->] (3) to [out=45,in=-45,looseness=5] (3);
\end{tikzpicture} 
\end{center}
under the action of $f$.  Since $\alpha_1$ is a root of the degree $d+1$ equation $f(x)=x$, we have degree bounds
\begin{equation*}
\begin{split}
[K(\alpha_1):K] & \leq d+1 \\
[K(\alpha_1,\alpha_2):K(\alpha_1)] & \leq d \\
[K(\alpha_1,\alpha_2,\alpha_3):K(\alpha_1,\alpha_2)] & \leq d
\end{split}
\end{equation*}
from which we deduce that $[K(\alpha_1,\alpha_2,\alpha_3):K]\leq d^2(d+1)$.  Setting $\beta_i=\varphi^{-1}(\alpha_i)$, we have $[K(\beta_1,\beta_2,\beta_3):K]\leq d^2(d+1)$ as well, as the $\beta_i$ have the same directed graph under $g$ as the $\alpha_i$ have under $f$.

Set $L=K(\alpha_1,\alpha_2,\alpha_3,\beta_1,\beta_2,\beta_3)$; thus $[L:K]\leq d^4(d+1)^2$.  Since the $\alpha_i$ are distinct and the automorphism $\varphi$ takes $\{\beta_1,\beta_2,\beta_3\}$ in $\PP^1(L)$ to $\{\alpha_1,\alpha_2,\alpha_3\}$ in $\PP^1(L)$, we can conclude that $\varphi$ is defined over $L$, as we can express $\varphi$ using cross-ratios.

Suppose now that $f$ and $g$ are polynomials and that $\varphi(x)=ax+b\in\Kbar[x]$ is an affine automorphism.  Then we only need distinct points $\alpha_1,\alpha_2\in \Kbar$ such that $\alpha_1$ is fixed by $f$ and $f(\alpha_2)=\alpha_1$.  A similar argument shows that $L=K(\alpha_1,\alpha_2,\beta_1,\beta_2)$ has degree $\leq d^2(d+1)^2$ over $K$, and $\varphi$ is defined over $L$ as it takes the triple $\{\infty,\beta_1,\beta_2\}$ to $\{\infty,\alpha_1,\alpha_2\}$.
\end{proof}

\begin{rem}
We do not know if the bounds described in Lemma \ref{DegreeBoundLemma} are best-possible. It would be an interesting problem to try to optimize this result, but we do not pursue this here.  
\end{rem}

Given polynomials $f_1(x),\dots,f_m(x)\in\Qbar[x]$ and points $\alpha_1,\dots,\alpha_n\in\Qbar$, we denote by $\QQ(f_1,\dots,f_m,\alpha_1,\dots,\alpha_n)$ the number field generated over $\QQ$ by the $\alpha_i$ and the coefficients of the $f_i$.  

Given a pair $(f,\alpha)\in\Qbar[x]\times\Qbar$, recall that $(f,\alpha)$ is a {\em potentially abelian arboreal pair} if there exists a finite extension $K/\QQ(f,\alpha)$ for which $\Gal(K_\infty(f,\alpha)/K)$ is abelian.

\begin{proof}[Proof of Theorem \ref{IntroSparsityTheorem}]
Let $d\geq2$ be an integer.  For each $T>0$, define $\Acal_{T}$ to be the set of all potentially abelian arboreal pairs $(f,\alpha)\in \Qbar[x]\times\Qbar$ with $\deg(f)=d$ such that  $[\QQ(f,\alpha):\QQ]\leq T$.  We must prove that $\Acal_{T}$ is contained in a finite union of $\Qbar$-conjugacy classes in $\Qbar[x]\times\Qbar$.

Let $\Bcal_{T}$ be the set of pairs in $(f,\alpha)\in \Qbar[x]\times\Qbar$ such that 
\begin{itemize}
\item $\deg(f)=d$, 
\item $f$ is PCF, 
\item $\alpha$ is $f$-preperiodic,
\item $[\QQ(f,\alpha):\QQ]\leq T$.
\end{itemize}
It follows from Theorem \ref{MainThmIntro} combined with Theorem A of Ferraguti-Ostafe-Zannier \cite{MR4686319} that for any extension $K/\QQ(f,\alpha)$, if $\Gal(K_\infty(f,\alpha)/K)$ is abelian then $f$  must be PCF and $\alpha$ must be $f$-preperiodic; therefore $\Acal_{T}\subseteq \Bcal_{T}$.  So it suffices to show that $\Bcal_{T}$ is contained in a finite union of $\Qbar$-conjugacy classes in $\Qbar[x]\times\Qbar$.

Ingram \cite{MR2885981} has proved that PCF polynomials in $\Qbar[x]$ of degree $d$ are a bounded height family in an appropriate moduli space, which implies in particular that there exists a finite list $f_1,f_2,\dots,f_M$ of PCF polynomials of degree $d$ in $\Qbar[x]$ such that every PCF polynomial $f(x)\in\Qbar[x]$ of degree $d$ and $[\QQ(f):\QQ]\leq T$ is $\Qbar$-conjugate to some $f_m$.

Now let $(f,\alpha)\in\Bcal_{T}$ be arbitrary.  By the preceding paragraph, we have that $f_m=\varphi^{-1}\circ f\circ \varphi$ for some $1\leq m\leq M$ and some affine automorphism $\varphi(x)=ax+b\in \Kbar[x]$ with $a\neq0$.  Taking $K=\QQ(f,f_m,\alpha)$, we may assume that $\varphi$ is defined over an extension $L/K$ of degree $[L:K]\leq C_d=d^2(d+1)^2$ using Lemma \ref{DegreeBoundLemma}.  Set $\beta=\varphi^{-1}(\alpha)\in\PP^1(L)$.   Since $\alpha$ is $f$-preperiodic, it follows that $\beta$ is $f_m$-preperiodic.  

The height of $\beta$ is bounded, depending on $d$ and $T$, but otherwise independently of the pair $(f,\alpha)$.  Indeed, since preperiodic points for a given polynomial are a set of bounded height, we obtain that $h(\beta)\leq H_m$ for some constant $H_m>0$ depending only on the polynomial $f_m$, and hence $h(\beta)\leq H:=\max_{1\leq m\leq M}H_m$.  The polynomials $f_m$, and hence the bound $H$, depend on $d$ and $T$ but not on the pair $(f,\alpha)$.

We can also bound the degree $[\QQ(\beta):\QQ]$, depending on $d$ and $T$, but otherwise independently of the pair $(f,\alpha)$.  Using standard properties of the degrees of field extensions in towers and in compositum, we have
\begin{equation*}
\begin{split}
[L:\QQ] & = [L:K][K:\QQ] \\
	& \leq C_d[\QQ(f,\alpha, f_m):\QQ] \\
	& \leq C_d[\QQ(f,\alpha):\QQ][\QQ(f_m):\QQ] \\
	& \leq C_dTD \\
\end{split}
\end{equation*}
where $D=\max_{1\leq m\leq M}[\QQ(f_m):\QQ]$.  Since $\beta\in\PP^1(L)$ we obtain $[\QQ(\beta):\QQ]\leq C_dTD$ as well.

Since the height and degree of $\beta$ have been bounded depending only on $d$ and $T$, which are fixed, there can be only finitely many possibilities for $\beta$.  Hence there are only finitely many possibilities for the pair $(f_m,\beta)$, completing the proof that $\Bcal_{T}$ is contained in a finite union of $\Qbar$-conjugacy classes in $\Qbar[x]\times\Qbar$.
\end{proof}



\def\cprime{$'$}

\end{document}